\documentclass[a4paper,oneside,11pt]{article}

\usepackage{amsmath,amsfonts,amscd,amssymb}
\usepackage{longtable,geometry}
\usepackage[english]{babel}
\usepackage[utf8]{inputenc}
\usepackage[active]{srcltx}
\usepackage[T1]{fontenc}
\usepackage{graphicx}
\usepackage{pstricks}
\usepackage{bbm}
\usepackage{mathtools}
\usepackage{enumitem}
\usepackage{MnSymbol}
\usepackage{stmaryrd}
\usepackage{nicefrac}
\usepackage{calrsfs}

 \usepackage{rotating}

\usepackage{xcolor}
\usepackage{framed}

\colorlet{shadecolor}{blue!15}

\newtheorem{theorem}{Theorem}[section]

\newtheorem{lemma}[theorem]{Lemma}

\newtheorem{remark}[theorem]{Remark}

\newenvironment{proof}[1][\relax]
  {\paragraph{Proof\ifx#1\relax\else~of #1\fi}}%
  {~\hfill$\square$\par\bigskip}

\newcommand{\bbE}{\mathbb{E}}

\newcommand{\bbG}{\mathbb{G}}
\newcommand{\bbH}{\mathbb{H}}

\newcommand{\bbR}{\mathbb{R}}

\newcommand{\bbV}{\mathbb{V}}

\newcommand{\bbZ}{\mathbb{Z}}

\newcommand{\U}{{\mathbf U}}
\newcommand{\V}{{\mathbf V}}

\newcommand{\e}{{\bf e}}
\renewcommand{\U}{{\bf U}}
\newcommand{\X}{{\bf X}}

\newcommand{\Y}{{\bf Y}}

\usepackage{tabularx}
\usepackage{graphicx}

\usepackage{constants}
\newconstantfamily{small}{
symbol=c}

\title{Sharp phase transition for the random-cluster and Potts models via decision trees}

\author{Hugo Duminil-Copin\thanks{Universit\'e de Gen\`eve} \thanks{Institut des Hautes \'Etudes Scientifiques} , Aran Raoufi\addtocounter{footnote}{-1}\footnotemark\ , Vincent Tassion\thanks{ETH Zurich}}
\date{\today}

\date{\today} 

\begin{document}
\maketitle

\begin{abstract}
We prove an inequality on decision trees on monotonic measures which generalizes the OSSS inequality on product spaces.
As an application, we use this inequality to prove a number of new results on lattice spin models and their random-cluster representations. More precisely, we prove that \begin{itemize}
  \item For the Potts model on transitive graphs, correlations decay exponentially fast for $\beta<\beta_c$.
  \item For the random-cluster model with cluster weight $q\ge1$ on transitive graphs, correlations decay exponentially fast in the subcritical regime and the cluster-density satisfies the mean-field lower bound in the supercritical regime.  
 \item For the random-cluster models with cluster weight $q\ge1$ on planar quasi-transitive graphs $\bbG$,
 $$\frac{p_c(\bbG)p_c(\bbG^*)}{(1-p_c(\bbG))(1-p_c(\bbG^*))}~=~q.$$
As a special case, we obtain the value of the critical point for the square, triangular and hexagonal lattices (this provides a short proof of the result of \cite{beffara2012self}).
\end{itemize}
These results have many applications for the understanding of the subcritical (respectively disordered) phase of all these models. The techniques developed in this paper have potential to be extended to a wide class of models including the Ashkin-Teller model, continuum percolation models such as Voronoi percolation and Boolean percolation, super-level sets of massive Gaussian Free Field, and random-cluster and Potts model with infinite range interactions.
\end{abstract}

\section{Introduction}

\subsection{OSSS inequality for monotonic measures}

In theoretical computer science, determining the computational complexity of tasks is a very difficult problem (think of P against NP). To start with a more tractable problem, computer scientists have studied {\em decision trees}, which are simpler models of computation. A decision tree associated to a Boolean function $f$ takes $\omega\in\{0,1\}^n$ as an input, and reveals algorithmically the value of $\omega$ in different bits one by one. The algorithm stops as soon as the value of $f$ is the same no matter the values of $\omega$ on the remaining coordinates. The question is then to determine how many bits of information must be revealed before the algorithm stops. The decision tree can also be taken at random to model random or quantum computers.

The theory of (random) decision trees played a key role in computer science  (we refer the reader to the survey \cite{buhrman2002complexity}), but also found many applications in other fields of mathematics. In particular, random decision trees (sometimes called randomized algorithms) were used in \cite{schrammsteif} to study the noise sensitivity of Boolean functions, for instance in the context of percolation theory. 

The OSSS inequality, originally introduced in \cite{OSSS} for product measure as a step toward a conjecture of Yao \cite{yao1977probabilistic}, relates the variance of a Boolean function to the influence of the variables and the computational complexity of a random decision tree for this function. The first part of this paper consists in generalizing the OSSS inequality to the context of monotonic measures which are not product measures.  
A {\em monotonic} measure is a measure $\mu$ on $\{0,1\}^E$ such that for any $e\in E$, 
any $F \subset E$, and any $\xi,\zeta\in \{0,1\}^{F}$ satisfying $\xi \le  \zeta $, $\mu[\omega_e=\xi_e,\forall e\in F]>0$ and  $\mu[\omega_e=\zeta_e,\forall e\in F]>0$,   
$$\mu[\omega_e=1 \: |\: \omega_e=\xi_e,\forall e\in F]\le \mu[\omega_e=1 \: |\: \omega_e=\zeta_e,\forall e\in F].$$

The motivation to choose such a class of measures comes from the applications to mathematical physics (for example, any positive measure satisfying the FKG-lattice inequality is monotonic, see \cite{Gri06} for more details), but monotonic measures also appear in computer science.

In order to state our theorem, we introduce a few notation. Consider a finite set $E$ of cardinality $n$. For a $n$-tuple $e=(e_1,\dots,e_n)$ and $t\le n$, write $e_{[t]}=(e_1,\dots,e_t)$ and $\omega_{e_{[t]}}=(\omega_{e_1},\dots,\omega_{e_t})$.

A decision tree encodes an algorithm that takes $\omega \in \{0,1\}^E$ as an input, and then queries the values of $\omega_e$, $e\in E$ one bits 
after the other. For any input $\omega$, the algorithm always starts from the same fixed $e_1\in E$ (which corresponds to the root of the decision 
tree), and queries the value of $\omega_{e_1}$. Then, the second element $e_2$ examined by the algorithm is prescribed by the decision tree and  
may depend on the value of  $\omega_{e_1}$. After having queried the value of $\omega_{e_2}$, the algorithm continues inductively. At step $t>1$,  
$(e_1,\ldots,e_{t-1})\in E^{t-1}$ has been examined, and the values of $\omega_{e_1},\ldots,\omega_{e_{t-1}}$ have been queried. The next element 
$e_t$ to be examined by the algorithm is a deterministic function of what has been explored in the previous steps:
\begin{equation}
e_t=\phi_t\big((e_1,\ldots,e_{t-1}), \omega_{(e_1,\ldots,e_{t-1})})\in E\setminus \{e_1,\dots,e_{t-1}\}.\label{eq:1}
\end{equation}

($\phi_t$ should be interpreted as the decision rule at time $t$: $\phi_t$ takes the location and the value of the first $t-1$ steps of the induction, and decides of the next bit to query).
Formally, we call  {\em decision tree} a pair $T=(e_1, (\phi_t)_{2 \leq  t \leq n})$, where $e_1 \in E$, and for each $t$ the function $\phi_t$, as above, takes a pair  $((e_1,\ldots,e_{t-1}), \omega_{(e_1,\ldots,e_{t-1})})$ as an input and return an element $e_t\in E\setminus\{e_1,\dots,e_{t-1}\}$.     

Let $T=(e_1, (\phi_t)_{2 \leq  t \leq n})$ be a decision tree and $f:\{0,1\}^E\rightarrow \bbR$. Given $\omega\in\{0,1\}^E$ we consider the  $n$-tuple $(e_1,\ldots,e_n)$ defined inductively by \eqref{eq:1} (this corresponds to the ordering on $E$ that we get when we run the algorithm $T$ starting from the input $\omega$). We define
\begin{equation}
\label{eq:ddd}\tau(\omega)=\tau_{f,T}(\omega):=\min\big\{ t\ge1:\forall \omega'\in\{0,1\}^E,\quad\omega'_{e_{[t]}}=\omega_{e_{[t]}}\Longrightarrow f(\omega)=f(\omega') \big\}.
\end{equation}
In computer science, a decision tree is usually associated directly to a Boolean function $f$ and defined as a rooted directed tree in which internal nodes are labeled by elements of $E$, leaves by possible outputs, and edges are in correspondence with the possible values of the bits at vertices (see \cite{OSSS} for a formal definition). In particular, the decision trees are usually defined up to $\tau$, and not later on. In this paper, we chose the slightly different formalism described above, which is equivalent to the classical one, since it will be more convenient for the proof of the following theorem. 

\begin{theorem}
\label{thm:OSSS}
Fix an increasing function $f:\{0,1\}^E\longrightarrow [0,1]$ on a finite set $E$. For any monotonic measure $\mu$ and any decision tree $T$,  \begin{equation}
    \label{eq:OSSS}
 \mathrm{Var}_\mu(f)~\le~    \sum_{e\in E}  \delta_e(f,T) \, \mathrm{Cov}_\mu (f,\omega_e) ,
  \end{equation}
  where 
$
\delta_e(f,T):=\mu\big[\exists \, t \le \tau(\omega)\::\:e_t=e\big]
$
is the revealment (of $f$) for the decision tree $T$.

\end{theorem}

A slightly stronger form of this result is stated in Section 2. In this paper, we focus on applications of the previous result to statistical physics but we expect it to have a number of applications in the context of the theory of Boolean functions. The interested reader is encouraged to consult \cite{Odonnell14} for a detailed introduction to the subject.   Theorems regarding Boolean functions have already found several applications in statistical physics, especially in the context of the noise sensitivity. For a review of the relationship between percolation theory and the analysis of Boolean functions we refer the reader to the book of Garban and Steif \cite{garban2014noise}. 
\subsection{Sharpness of the phase transition in statistical physics}

We call {\em lattice} a locally finite (vertex-)transitive infinite graph $\bbG=(\bbV,\bbE)$. An (unoriented) edge of the lattice is denoted $xy$. We also distinguish a vertex $0\in\bbV$ and call it the origin. Let $d(\cdot,\cdot)$ denote the graph distance on $\bbG$. Introduce a family of {\em non-negative} coupling constants $J=(J_{xy})_{xy\in \bbE}\in[0,\infty)^\bbE$ which is non-zero and invariant under a group acting transitively on $\bbV$. Notice that the coupling constants are necessarily finite-range (since the graph is locally finite). We call the pair $(\bbG,J)$ a {\em weighted lattice}.

Statistical physics models defined on a lattice are useful to describe  
a large variety of phenomena and objects, ranging from ferro-magnetic materials to lattice gas. They also provide discretizations of Euclidean and Quantum Field Theories and are as such important from the point of view of theoretical physics. While the original motivation came from physics, they appeared as extremely complex and rich mathematical objects, whose study required the development of important new tools that found applications in many other domains of mathematics. 

One of the key aspects of these models is that they often undergo order/disorder phase transitions at a certain critical parameter $\beta_c$. The regime $\beta<\beta_c$, usually called the disorder regime, exhibits very rapid decay of correlations. While this property is usually simple to derive for very small values of $\beta$ using perturbative techniques, proving such a statement for the whole range of parameters $\beta<\beta_c$ is a difficult mathematical challenge. Nevertheless, having such a property is the key towards a deep understanding of the disordered regime. 

The zoo of lattice models is very diverse: it includes models of spin-glasses, quantum chains, random surfaces, spin systems, percolation models. One of the most famous example of a lattice spin model is provided by the Ising model introduced by Lenz to explain Curie's temperature for ferromagnets. This model has been generalized in many directions to create models exhibiting a wide range of critical phenomena. While the Ising model is very well understood, most of these natural generalizations remain much more difficult to comprehend. In this paper, we prove that the Potts model (one of the most natural of such generalization) undergoes a sharp phase transition, meaning that in the disordered regime, correlations decay exponentially fast. 
In order to do so, we will study the random-cluster representations of these models, which are often monotonic. The generalized OSSS inequality proved in Theorem~\ref{thm:OSSS} will play a key role in the proof. 

\paragraph{Exponential decay for the subcritical random-cluster model}

Since random-cluster models were introduced by Fortuin and Kasteleyn in 
1969~\cite{ForKas72}, they have become the archetypal example of dependent percolation models and as such have played an important role in the 
study of phase transitions. The spin correlations of Potts models are 
rephrased as cluster connectivity properties of their random-cluster 
representations. This allows the use of geometric techniques, thus 
leading to several important applications. While the understanding of the model on planar graphs progressed greatly in the past few years \cite{beffara2012self, duminil2017continuity, duminil2016discontinuity, DumRaoTas16}, the case of higher dimensions remained poorly understood. We refer to \cite{Gri06, Dum13} for books on the subject and a discussion of existing results.

The model is defined as follows. 
Consider a finite subgraph $G=(V,E)$ of a weighted lattice $(\bbG, J)$ and introduce the boundary $\partial G$ of $G$ to be the set of vertices $x\in G$ for which there exists $y\notin G$ with $xy$ an edge of $\bbE$. A {\em percolation configuration} $\omega=(\omega_{xy})_{xy\in E}$ is an element of $\{0,1\}^{E}$. A configuration $\omega$ can be seen as a subgraph of $G$ with vertex-set $V$ and edge-set given by $\{xy\in E:\omega_{xy}=1\}$. Let $k_{\rm f}(\omega)$ (resp.~$k_{\rm w}(\omega)$) be the number of connected components in $\omega$ (resp.~in the graph obtained from $\omega$ by considering all the vertices in $\partial G$ as one single vertex). 

Fix $q,\beta> 0$.
For $\#\in\{{\rm f},{\rm w}\}$, let $\phi_{G,\beta,q}^\#$ be the measure satisfying, for any $\omega\in\{0,1\}^E$,
$$\phi_{G,\beta,q}^\#(\omega)=\frac{q^{k_\#(\omega)}}{Z}\prod_{xy\in E}\big(e^{\beta J_{xy}}-1\big)^{\omega_{xy}},$$ 
where $Z$ is a normalizing constant introduced in such a way that $\phi_{G,\beta,q}^\#$ is a probability measure. The measures $\phi_{G,\beta,q}^{\rm f}$ and $\phi_{G,\beta,q}^{\rm w}$ are called the random-cluster measures on $G$ with respectively free and wired boundary conditions.
 For $q\geq 1$, the 
measures $\phi_{G,\beta,q}^\#$ can be extended to $\bbG$ -- the corresponding measure is denoted by $\phi_{\bbG,\beta,q}^\#$ -- by taking the weak limit of measures defined in finite volume. 

For notational convenience, we set $x\longleftrightarrow y$ if $x$ and $y$ are in the same connected component. We also write $x\longleftrightarrow Y$ if $x$ is connected to a vertex in $Y\subset \bbV$, and $x\longleftrightarrow \infty$ if the connected component of $x$ is infinite. Finally, let $\Lambda_n$ be the box of size $n$ around $0$ for the graph distance.

For $q\ge1$, the model undergoes a phase transition: there exists $\beta_c=\beta_c(\bbG)\in[0,\infty]$ satisfying
$$\theta(\beta):=\phi_{\bbG,\beta,q}^{\rm w}[0\longleftrightarrow\infty]=\begin{cases}=0&\text{ if $\beta<\beta_c$,}\\ >0&\text{ if $\beta>\beta_c$.}\end{cases}$$

The main theorem of this article is the following one.
\begin{theorem}\label{thm:RCM}
Fix $q \geq 1$ and consider the random-cluster model on a weighted lattice $(\bbG, J)$.
Then,
\begin{itemize}[noitemsep]
\item There exists $c>0$ such that $\theta(\beta)\ge c(\beta-\beta_c)$ for any $\beta\ge\beta_c$ close enough to $\beta_c$.
\item For any $\beta<\beta_c$ there exists $c_\beta>0$ such that for every $n\ge0$,
$$\phi_{\Lambda_n,\beta,q}^{\rm w}[0\longleftrightarrow \partial\Lambda_n]\le \exp[-c_\beta n].$$
\end{itemize}
\end{theorem}

Theorem~\ref{thm:RCM} extends to quasi-transitive weighted graphs and to finite range interactions (for the latter, simply interpret finite-range models as nearest-neighbor models on a bigger graph).
 
For planar graphs, the result was proved for any $q\ge1$ under some symmetry assumption in \cite{duminil2016phase} (see also \cite{manolescu2016phase} for the case of planar slabs).
On $\bbZ^d$, the result was restricted to large values of $q$ \cite{laanait1991interfaces} and to the special cases of Bernoulli percolation ($q=1$) \cite{menshikov1986coincidence,aizenman1987sharpness, duminil2015new} and the FK-Ising model ($q=2$) \cite{AizBarFer87,duminil2015new}.
 
Numerous results about the subcritical regime have been proved under the assumption of exponential decay, and therefore Theorem~\ref{thm:RCM} transform them into unconditional results. To cite but a few, let us mention the Ornstein-Zernike theory of correlations \cite{campanino2008fluctuation}, the mixing properties of the model \cite{alexander2004mixing}, the bounds on the spectral gaps of the associated dynamics \cite{martinelli1999lectures}. The second item of Theorem~\ref{thm:RCM} could be replaced by $\phi_{\bbG,\beta,q}^{\rm w}[0\longleftrightarrow x]\le \exp[-c_\beta d(0,x)]$, but the stronger statement proved in the theorem is the one useful for these applications.

 \paragraph{Applications to computations of critical points for planar graphs} Another important application of Theorem~\ref{thm:RCM} is the computation of critical points of specific lattices.
In this section, we fix coupling constants to be equal to $1$ and set $p_c:=1-e^{-\beta_c}$. In general, the critical parameter $p_c$ is not expected to take any specific value. However, for the square, hexagonal and triangular lattices, the critical values can be predicted using duality. It is proved  in \cite[Theorem 6.17]{Gri06} that predicted values are indeed the critical ones under the assumption of exponential decay  for $p < p_c$. Therefore, our result provides an alternative proof of the following theorem. 
\begin{theorem}\label{thm:planarRCM}
  Fix $q\ge1$. If $y_c:=p_c/(1-p_c)$, we have  $$\begin{array}{ll}
    y_c^2-q=0&\text{on the square lattice,} \\
    y_c^3+3y_c^2-q=0  & \text{on the triangular lattice,}\\
    y_c^3-3qy_c-q^2=0 & \text{on the hexagonal lattice.}
  \end{array}$$
\end{theorem}
Note that for the square lattice, $p_c$ is equal to $\sqrt q/(1+\sqrt q)$.
This result was originally proved in \cite{beffara2012self}, where exponential decay of correlations below $p_c$ is proved using Russo-Seymour-Welsh type arguments and a generalization \cite{graham2006influence} of the KKL result \cite{kahn1988influence, bourgain1992influence}. 
 
The fact that our proof of exponential decay requires very few conditions on the graphs enables us to study critical points of general planar locally-finite doubly periodic graphs, i.e.~embedded planar graphs which are invariant under the action of some lattice $\Lambda\approx \bbZ\oplus \bbZ$. Denote  the dual of any planar graph $\bbG$ by $\bbG^*$.

 \begin{theorem}\label{thm:dualRCM}
  Fix $q\ge1$ and a planar locally-finite doubly periodic graph $\bbG$. We have  
\begin{equation}\label{eq:p}\frac{p_c(\bbG)p_c(\bbG^*)}{(1-p_c(\bbG))(1-p_c(\bbG^*))}= q.\end{equation}
\end{theorem}
 This result should be understood as a generalization of the famous statement $p_c(\bbG)+p_c(\bbG^*)=1$ for Bernoulli percolation.
The theorem is a consequence of duality, exponential decay for $p<p_c(\bbG)$ and the following non-coexistence result. For a configuration $\omega$ on $\bbG$, define a configuration $\omega^*$ in $\bbG^*$ by the formula $\omega^*_{e^*}=1-\omega_e$ for every edge $e$ of $\bbG$, where $e^*$ is the edge of $\bbG^*$ between the two vertices of $\bbG^*$ corresponding to the faces bordered by $e$.
\begin{theorem}\label{thm:dual}
There does not exist any translational invariant measure $\mu$ on a planar locally-finite doubly periodic graph  $\bbG$ satisfying 
\begin{itemize}[noitemsep,nolistsep]
\item {\em (FKG)} For any increasing events $A$ and $B$,
$\mu(A\cap B)\ge \mu(A)\mu(B)$.
\item Almost surely, there exists a unique infinite connected component in $\omega$ and in $\omega^*$. 
\end{itemize}
\end{theorem}
This result was proved in \cite{sheffield2005random} . It was also proved for percolation on self-dual polygon configurations in \cite{bollobas2010percolation}. Here, we present a proof which has also the advantage of being quite short.

\paragraph{Applications to the ferromagnetic $q$-state Potts model}
The Potts model \cite{Pot52} is one of the most fundamental example of a lattice spin model undergoing an order/disorder phase transition at a critical parameter $\beta_c$. It generalizes the Ising model by allowing the spins to take one of $q$ values. In two dimensions, the model has been the object of intense study in the past few years and  the behavior is fairly well understood, even at criticality \cite{ duminil2017continuity, duminil2016discontinuity}. In higher dimension, the understanding is limited to the case of the Ising model (i.e.~$q=2$) and of large $q$ \cite{AizFer86, aizenman2015random,KotShl82, laanait1991interfaces, biskup2003rigorous}. 

The model is defined as follows. Consider an integer $q\ge2$. For $G=(V,E)$ a finite subgraph of a weighted lattice $(\bbG, J)$, $\nu \in \{1, \dots, q\}$, and $\beta\ge0$, the {\em $q$-state Potts measure} with boundary condition $\nu$ is defined for any $\sigma=(\sigma_x)_{x\in V}\in \{1,\dots,q\}^V$ by
$$ \mathbb P^{\,\nu}_{G,\beta,q}[\sigma]~:=~\frac{\displaystyle \exp  ( -\beta H^\nu_{G,q}(\sigma) )}{\displaystyle\sum_{\sigma'\in\{1,\dots,q\}^V}\exp  ( -\beta H^\nu_{G,q}(\sigma) )}$$
where
$$H^\nu_{G,q}(\sigma)~:= - ~\sum_{xy\in E} J_{xy} \,\delta_{\sigma_x = \sigma_y} -    \sum_{\substack{xy\in E \\ x \in \partial G, \, y \notin G}} J_{xy} \,\delta_{\sigma_x = \nu} .$$

The model can be defined in infinite volume by taking the weak limit of measures on a nested sequence of finite graphs. The obtained measure $\mathbb P^\nu_{\bbG,\beta,q}$ is called the Potts measure with boundary conditions $\nu$.  The Potts model undergoes a phase transition between absence and existence of long-range order at the so-called critical inverse temperature $\beta_c$ (which depends on $\bbG$ and $J$), see  \cite{Gri06} for details.

\begin{theorem}\label{thm:Potts}
Fix an integer $q\ge2$ and consider the $q$-state Potts on a weighted lattice $(\bbG,J)$. Then, for $\beta<\beta_c$,  there exists $c_\beta>0$ such that for every $x\in\bbV$,
$$0\le\mathbb P^\nu_{\Lambda_n,\beta,q}[\sigma_0=\nu]-\tfrac1q\le \exp[-c_\beta n].$$
Furthermore, for the nearest-neighbor model on the square lattice,
$\beta_c(\bbZ^2)=\log(1+\sqrt q).$
\end{theorem}
For the $2$-state Potts model, better known as the Ising model, the result goes back to \cite{AizBarFer87} (see also \cite{duminil2015new}). For the $q$-state Potts model with $q\ge3$, the result was restricted to either perturbative arguments involving the Pirogov-Sinai theory for $q\gg1$ or planar arguments (see the discussion on the random-cluster model). The question of deriving this property for $q\ge 3$ and $\bbZ^d$ with $d\ge3$ was open. 
Again, the flexibility in the choice of the lattice $\bbG$ implies that the result applies to finite range interactions. The statement of Theorem \ref{thm:Potts} is  stronger  than the statement  $\mathbb P^\nu_{\bbG,\beta,q} [\sigma_0 =  \nu] - \tfrac{1}{q} \le \exp[-c_\beta d (0,x)]$.

The Potts model and the random-cluster models on a weighted lattice $(\bbG,J)$ can be coupled (see \cite[Theorem 1.10]{Gri06} for details) in such a way that 
$$\mathbb P^\nu_{\Lambda_n,\beta,q}[\sigma_0=\nu]-\tfrac1q=\tfrac{q-1}q\phi_{\Lambda_{n+1},\beta,q}^{\rm w}[0\longleftrightarrow \partial \Lambda_{n+1}],$$
so that Theorem~\ref{thm:Potts} is a direct consequence of Theorems~\ref{thm:RCM} and \ref{thm:planarRCM}.

\paragraph{Other models.} The reasoning above should extend to other lattice spin models for which there exists a random-cluster representation which is monotonic. An archetypal example is provided by the Ashkin-Teller model; see \cite{baxter1982exactly} for details. It also extends to continuum percolation models such as Voronoi percolation \cite{DumRaoTas17}, occupied and vacant set of Boolean percolation \cite{DumRaoTas17-2}, massive Gaussian free field super-level lines.

\paragraph{Organization} The paper is organized as follows. In the next section we prove Theorem~\ref{thm:OSSS}. In the third section, we prove Theorem~\ref{thm:RCM} (we tried to isolate a few general statements which may be used for the proof of exponential decay for other models of statistical physics). In the last section, we describe the proof of Theorems~\ref{thm:dualRCM} and \ref{thm:dual}.

\paragraph{Acknowledgments} 
This research was supported by an IDEX grant from Paris-Saclay, a grant from the Swiss FNS, the ERC CriBLaM, and the NCCR SwissMAP. We thank Yvan Velenik, Alain-Sol Sznitman and Ioan Manolescu for many inspiring discussions and insightful comments.

\section{Proof of Theorem~\ref{thm:OSSS}}

The strategy is a combination of the original proof of the OSSS inequality for product measures (which is an Efron-Stein type reasoning), together with an encoding of monotonic measures in terms of iid random variables. 
  Assume that $E$ is finite and has cardinality $n$. Let $\vec E$ be the set of sequences $e=(e_1,\dots,e_n)$ where each element of $E$ occurs exactly once. Consider a monotonic measure $\mu$ on $\{0,1\}^E$.  
  
  We start by a useful lemma explaining how to construct $\omega$ with law $\mu$ from iid uniform random variables.
For $u\in[0,1]^n$ and $e\in \vec E$, define $F_e(u)=x$ inductively for $1\le t\le n$ by
   \begin{equation}\label{eq:ccc}x_{e_t}:=\begin{cases} 1&\text{ if }u_{t}\ge\mu[\omega_{e_{t}}=0\,|\,\omega_{e_{[t-1]}}=x_{e_{[t-1]}}],\\
   0&\text{ otherwise}.\end{cases}\end{equation}

\begin{lemma}\label{lem:01}
Let $\U$ be a iid sequence of uniform $[0,1]$ random variables, and $\e$ a random variable taking values in $\vec E$. Assume that for every $1\le t\le n$, $\U_t$ is independent of $(\e_{[t]},\U_{[t-1]})$, then $\X=F_{\e}(\U)$ has law $\mu$.
\end{lemma}
  
\begin{proof}
Let  $x\in \{0,1\}^E$ and $e\in \vec E$ such that $\mathbb P[\X=x,\e=e]>0$. The probability $\mathbb P[\X=x,\e=e]$ can be written as
$$\prod_{t=1}^n\mathbb P[\X_{e_t}=x_{e_t}\,|\,\e_{[t]}=e_{[t]},\X_{e_{[t-1]}}=x_{e_{[t-1]}}]\times \prod_{t=1}^n\mathbb P[\e_t=e_t\,|\,\e_{[t-1]}=e_{[t-1]},\X_{e_{[t-1]}}=x_{e_{[t-1]}}].$$
(All the conditionings are well defined, since we assumed $\mathbb P[\X=x,\e=e]>0$.)
Since $\U_t$ is independent of $(\e_{[t]}, U_{[t-1]})$ (and thus $\X_{e_{[t-1]}}$), the definition \eqref{eq:ccc} gives 
$$\mathbb P[\X_{e_t}=x_{e_t}\,|\,\e_{[t]}=e_{[t]},\X_{e_{[t-1]}}=x_{e_{[t-1]}}]=\mu[\omega_{e_t}=x_{e_t}\,|\,\omega_{e_{[t-1]}}=x_{e_{[t-1]}}]$$
so that the first product is equal to $\mu[\omega=x]$ independently of $e$. Fixing $x\in \{0,1\}^E$, and summing on $e\in \vec E$ satisfying  $\mathbb P[\X=x,\e=e]>0$ gives
\begin{align*}
\mathbb P[\X=x]&=\sum_{e}\mathbb P[\X=x,\e=e]\\
&=\mu[\omega=x]\sum_{e}\prod_{t=1}^n\mathbb P[\e_t=e_t|\e_{[t-1]}=e_{[t-1]},\X_{e_{[t-1]}}=x_{e_{[t-1]}}]=\mu[\omega=x].
\end{align*}

\end{proof}

\begin{proof}[Theorem~\ref{thm:OSSS}]
Our goal is to apply a Lindenberg-type argument on a probability space in which $\e$ and $\X$ (sampled according to $\mu$) are coupled to an independent copy of $\X$  (denoted by $\Y^{n}$ below). We now present the coupling. 

Consider two independent sequences of iid uniform $[0,1]$ random variables $\U$ and $\V$. Write $\mathbb P$ for the coupling between these variables (and $\mathbb E$ for its expectation). 
Construct $(\e,\X,\tau)$ inductively as follows: set for $t\ge1$,  
\begin{align*}
 \e_{t}&=\begin{cases}
   e_1 & \text{if $t=1$}\\
   \phi_t(\e_{[t-1]},\X_{\e_{[t-1]}})&\text{if $t>1$}\end{cases}\quad\text{and}\quad \X_{\e_t}=\begin{cases} 1&\text{ if }\U_{t}\ge\mu(\omega_{\e_{t}}=0\,|\,\omega_{\e_{[t-1]}}=\X_{\e_{[t-1]}}),\\
   0&\text{ otherwise,}\end{cases}
  \end{align*}
and $\tau:=\min\big\{t\ge1:\forall x\in\{0,1\}^E,x_{\e_{[t]}}=\X_{\e_{[t]}}\Rightarrow f(x)=f(\X)\big\}$. Note that $\tau$ is equal to the stopping time defined in \eqref{eq:ddd}. Finally, for $0\le t\le n$, define $\Y^t:=F_\e(\mathbf W^t),$ where
$$\mathbf W^t:= \mathbf W^t (\U,\V) =(\V_1,\dots,\V_t,\U_{t+1},\dots,\U_\tau,\V_{\tau+1},\dots,\V_n)$$
(in particular $\mathbf{W} ^t$ is equal to $\V$ if $t\ge \tau$). 

Lemma~\ref{lem:01} applied to $(\U,\e)$ gives that $\X$ has law $\mu$ and is $\U$-measurable. Lemma~\ref{lem:01} applied to $(\V,\e)$  implies that $\Y^{n}$ has law $\mu$ and is independent of $\U$. Therefore, using that $f$ is valued in $[0,1]$, we deduce that
\begin{equation*}\label{eq:a}
{\rm Var}_\mu(f)\le \tfrac12\mu\big[|f-\mu[f]|\big]=\tfrac12\bbE\Big[\big|\,\bbE[f(\X)|\U]-\bbE[f(\Y^{n})|\U]\,\big|\Big]\le\tfrac12\mathbb E\big[|f(\X)-f(\Y^{n})|\big].
\end{equation*}
Since $f(\Y^{0})=f(\X)$ (the entries of $\Y^0$ for $t>\tau$ are irrelevant for the value of $f$ by definition of~$\tau$), the equation above implies
\begin{equation*}
  \label{eq:2}
  {\rm Var}_\mu(f)\le \tfrac12\mathbb E\big[|f(\Y^{0})-f(\Y^{n})|\big].
\end{equation*}
Since $\Y^t=\Y^{t-1}$ for any $t>\tau$, the right-hand side of the previous inequality is less than or equal to
\begin{align*}\sum_{t=1}^{n}\mathbb E \big[|f(\Y^t)-f(\Y^{t-1})|\big]&= \sum_{t=1}^{n}\mathbb E\Big[\,|f(\Y^t)-f(\Y^{t-1})|\,\cdot\mathbbm1_{t\le \tau}\Big]\\
&=\sum_{e\in E} \sum_{t=1}^{n}\mathbb E\Big[\mathbb E\big[|f(\Y^t)-f(\Y^{t-1})| ~\big|~\U_{[t-1]}\big]\,\mathbbm1_{t\le \tau,\e_t=e}\Big].
\end{align*}
Recalling that $\sum_{t=1}^{n}\mathbb P[t\le \tau,\e_t=e]=\delta_e(f,T),$
the proof of the theorem follows from the fact that on $\{t\le \tau,\e_t=e\}$,
\begin{equation}\mathbb E\big[\,|f(\Y^t)-f(\Y^{t-1})| ~\big|~\U_{[t-1]}\big]~\le~ 2\mathrm{Cov}_\mu (f, \omega_{e} ).\label{eq:g}\end{equation}
In order to show this, we now restrict ourself to the event $\{t\le \tau,\e_t=e\}$. First observe that $\Y^t _{e}= \Y^{t-1}_{e}$ implies $\Y^t = \Y^{t-1}$, and this together with the fact that $f$ is increasing implies
\begin{align}
|f(\Y^t)-f(\Y^{t-1})| &=(f(\Y^t)-f(\Y^{t-1}) ) \,(\Y^t_{e} - \Y^{t-1}_{e} )\nonumber\\
&=f(\Y^{t-1})\Y^{t-1}_{e}+  f(\Y^t)\Y^t_{e}- f(\Y^{t-1})\Y^t_{e}- f(\Y^t)\Y^{t-1}_{e}.\label{eq:tt}
\end{align}
Our goal is to average against $\bbE[\cdot|\U_{[t-1]}]$. In order to do this, we will use the following claim.\medbreak
\noindent
{\em {\bf Claim.} For any measurable $g$ and $t\le n$,}
\begin{equation}\label{eq:bbb} \mathbb E[g(\Y^t)|\U_{[t]}]=\mu[g(\omega)].\end{equation} 
\begin{proof} Conditioned on $\U_{[t]}$, the random vector
$\mathbf W^t$
is composed of iid uniform random variables satisfying that $\mathbf W^t_i$ is independent of $(\e_1,\dots,\e_i)$ for every $i\le n$. Therefore, Lemma~\ref{lem:01} applied to $(\e,\mathbf W^t)$ implies that the law of $\Y^t$ conditioned on $\U_{[t]}$ is $\mu$, which gives the claim.\end{proof} 
Applying \eqref{eq:bbb} to $g(\omega)=f(\omega)\omega_{e}$ gives that (for the second equality we average on $\U_t$)
\begin{align} \label{eq:a1}
\bbE[f( \Y^{t-1}) \Y^{t-1}_{e} \,|\, \U_{[t-1]}]= \mu[f(\omega)\omega_{e}]=\bbE[f( \Y^t) \Y^t_{e} \,|\, \U_{[t-1]}].
\end{align}
For fixed $\U_{[n]}$ and $s$, $\Y^s=F_\e(\mathbf{W}^s)$ is an increasing function of $\V$, by monotonicity of $\mu$. Since $f$ and $\mathbf W_e$ are increasing functions of $\mathbf V$, we deduce that $f(\Y^{t-1})$ and $\Y^t_e=F_\e(\mathbf{W})_e$ are increasing functions of $\V$. The FKG inequality applied to the  iid random variables $\V$ gives
\begin{align*}
\bbE[f(\Y^{t-1})\Y^t_{e}|\U_{[n]}]&\ge \bbE[f(\Y^{t-1})|\U_{[n]}]\bbE[\Y^t_{e}|\U_{[n]}].
\end{align*}
Taking the expectation with respect to $\bbE[\,\cdot\,|\U_{[t-1]}]$ gives
\begin{align}\label{eq:a2}
\bbE[f(\Y^{t-1})\Y^t_{e}|\U_{[t-1]}]&\ge \bbE\big[\bbE[f(\Y^{t-1})|\U_{[n]} \big]\bbE \big[\Y^t_{e}|\U_{[n]}]\,\big|\,\U_{[t-1]}\big]\nonumber\\
&= \bbE \big[f(\Y^{t-1})|\U_{[t-1]}\big]\bbE\big[\Y^t_{e}|\U_{[t-1]}\big]\stackrel{\eqref{eq:bbb}}=\mu[f(\omega)]\mu[\omega_{e}],
\end{align}
where we used that $\bbE[\Y^t_{e}|\U_{[n]}]$ is $\U_{[t-1]}$-measurable (since $\Y^{t}_{e}$ depends on $\U_{[t-1]}$ and $\V$ only). 

Similarly, $f(\Y^t)$ and $\Y^{t-1}_{e}$ are increasing functions of $\V$ so that using the FKG inequality and then taking the expectation with respect to $\bbE[\,\cdot\,|\U_{[t]}]$ gives
\begin{align*}\bbE[f(\Y^t)\Y^{t-1}_{e}|\U_{[t]}]&\ge \bbE\big[\bbE[f(\Y^t)|\U_{[n]}]\bbE[\Y^{t-1}_{e}|\U_{[n]}]\,\big|\,\U_{[t]}\big]\\
&= \bbE[f(\Y^t)|\U_{[t]}]\bbE[\Y^{t-1}_{e}|\U_{[t]}]\stackrel{\eqref{eq:bbb}}=\mu[f(\omega)]\bbE[\Y^{t-1}_{e}|\U_{[t]}].\end{align*}
This time, we used that $\bbE[\Y^{t-1}_{e}|\U_{[n]}]$ is $\U_{[t]}$-measurable. 
Taking the expectation with respect to  $\bbE[\,\cdot\,|\U_{[t-1]}]$ gives
$$\bbE[f(\Y^t)\Y^{t-1}_{e}|\U_{[t-1]}] \ge \mu[f(\omega)]\bbE[\Y^{t-1}_{e}|\U_{[t-1]}]\stackrel{\eqref{eq:bbb}}=\mu[f(\omega)]\mu[\omega_{e}].$$
This inequality together with \eqref{eq:a2}, \eqref{eq:a1} and \eqref{eq:tt} give \eqref{eq:g} and therefore concludes the proof.
\end{proof}

\begin{remark}
For most applications, one may replace covariances in the OSSS inequality by influences $I_e[f]:=\mu (f|\omega_e=1)-\mu(f|\omega_e=0)$ (we chose not to do so since applications in statistical physics to long-range models would for instance require the statement with covariances). In this case, we do not need to prove \eqref{eq:g} anymore and can replace the lengthly end of the proof by the following short argument.
Recall the dependency in the measure $\mu$ in $F_e(u)$ and write $F_e^\mu(u)$. With this notation, one sees that $F$ is both increasing in $u$ and in $\mu$ (for stochastic domination).
We deduce that both $\Y^{t-1}$ and $\Y^t$ are sandwiched between $\mathbf Z:=F_{\e}^{\mu[\cdot|\omega_e=0]}(\mathbf W^t)$ and $\mathbf Z':=F_{\e}^{\mu[\cdot|\omega_e=1]}(\mathbf W^t)$.
Recall that $\mathbf W$ is independent of $\U_{[t-1]}$. Lemma~\ref{lem:01} and the fact that $f$ is increasing give us 
$$
\mathbb E\big[\,|f(\Y^t)-f(\Y^{t-1})| ~\big|~\U_{[t-1]}\big]\le \mathbb E[f(\mathbf Z')]-\mathbb E[f(\mathbf Z)]=\mu[f(\omega)|\omega_e=1]-\mu[f(\omega)|\omega_e=0]=I_e[f].$$
\end{remark}

\begin{remark}
Note that for the trivial decision tree discovering all the edges, for every edge the revealment is equal to 1 . As a consequence, we recover (in a very convoluted way) the discrete Poincar\'e inequality 
\begin{equation}
\label{eq:poincare}  \mathrm{Var}_\mu(f)~\le~   \sum_{e\in E}\, \mathrm{Cov}_\mu (f,\omega_e) .
\end{equation}
\end{remark}

\begin{remark}The proof of the previous statement can be extended in a trivial way as follows. First, we may consider countable sets $E$ by using a very simple martingale argument. Second, we may consider that $\tau$ is an arbitrary stopping time (with respect to the filtration $(\mathcal F_t=\sigma(\e_{[t]},\omega_{\e_{[t]}}))_{t\ge0}$), i.e.~that $f$ is not necessarily $\mathcal F_\tau$ measurable. By simply applying the previous lemma with $g=\mu[f|\mathcal F_\tau]$, we obtain the following result, which may be useful in statistical physics.
\end{remark}
\begin{theorem}
  Fix a countable set $E$ and an increasing function $f:\{0,1\}^E\longrightarrow [-1,1]$. For any monotonic measure $\mu$ on $\{0,1\}^E$, any decision tree $T$ and any stopping time $\tau$,  \begin{equation}
    \label{eq:OSSS2}
 \mathrm{Var}_\mu(f)~\le~    \sum_{e\in E}  \delta_e(f,T) \, \mathrm{Cov}_\mu (f,\omega_e)  ~+~\mu\big[|f-\mu[f| \mathcal F_ \tau]|\big]. 
  \end{equation}
\end{theorem}

    \section{Proof of Theorem~\ref{thm:RCM}}

  In order to be able to apply the strategy to other models, we state two useful lemmas.   
  \begin{lemma}\label{lem:technical}
Consider a converging sequence of increasing differentiable functions $f_n:[0,\beta_0]\rightarrow [0,M]$ satisfying
 \begin{equation}\label{eq:mlem}f_n'\ge \frac{n}{\Sigma_{n}}f_n
 \end{equation}
 for all $n\ge1$, where $\Sigma_n=\sum_{k=0}^{n-1}f_k$. Then, there exists $\beta_1\in[0,\beta_0]$ such that  
 \begin{itemize}
 \item[{\bf P1}] 
 For any $\beta<\beta_1$, there exists $c_\beta>0$ such that for any $n$ large enough,
 $f_n(\beta)\le \exp(-c_\beta n).$
  \item[{\bf P2}] 
  For any $\beta>\beta_1$, $\displaystyle f=\lim_{n\rightarrow \infty}f_n$ satisfies $f(\beta)\ge \beta-\beta_1.$
 \end{itemize}
 \end{lemma}
  
  \begin{proof}
Define 
$$ \beta_1 := \inf \Big\{ \beta \, : \, \limsup_{n \rightarrow \infty} \frac{\log \Sigma_n(\beta)}{\log n} \geq 1 \Big\}. $$

\paragraph{Assume $\beta<\beta_1$.} Fix $\delta>0$ and set $\beta'=\beta-\delta$ and $\beta''=\beta-2\delta$. We will prove that there is exponential decay at $\beta''$ in two steps. 

First, there exists an integer $N$ and $\alpha>0$  such that $\Sigma_n(\beta) \leq n ^ {1-\alpha}$ for all $n \geq N$. For such an integer $n$, integrating $f_n' \geq n^{\alpha} f_n$  between $\beta'$ and $\beta$ -- this differential inequality follows from \eqref{eq:mlem}, the monotonicity of the functions $f_n$ (and therefore $\Sigma_n$) and the previous bound on $\Sigma_n(\beta)$ -- implies that
$$ f_n(\beta') \leq M\exp(-\delta \,n^\alpha),\quad\forall n\ge N.$$

Second, this implies that there exists $\Sigma < \infty$ such that $\Sigma_n(\beta') \leq \Sigma$ for all $n$. Integrating $f_n' \geq \tfrac{n}{\Sigma} f_n$ for all $n$ between $\beta''$ and $\beta'$ -- this differential inequality is again due to \eqref{eq:mlem}, the monotonicity of $\Sigma_n$, and the bound on $\Sigma_n(\beta')$ -- leads to
\begin{equation*}f_n (\beta'') \leq M\exp(-\frac{\delta}{\Sigma} \,n),\quad\forall n\ge0.\label{eq:bb}\end{equation*}

\paragraph{Assume $\beta>\beta_1$.} For $n\ge 1$, define the function $T_n := \frac{1}{\log n} \sum_{i=1}^{n} \frac{f_i}{i}$. Differentiating $T_n$ and using \eqref{eq:mlem}, we obtain
$$T_n'~=~ \frac{1}{\log n}\, \sum_{i=1}^{n} \frac{f_i'}{i}  ~\stackrel{\eqref{eq:mlem}}{\geq}~ \frac{1}{\log n}\, \sum_{i=1}^{n} \frac{f_i}{\Sigma_{i}} ~\geq~ \frac{\log \Sigma_{n+1}-\log \Sigma_1}{\log n},$$
where in the last inequality we used that for every $i\ge1$,
$$\frac{f_i}{\Sigma_{i}} \geq \int_{\Sigma_{i}}^{\Sigma_{i+1}} \frac{dt}{t}=\log \Sigma_{i+1}-\log \Sigma_{i}.$$ 
For $\beta'\in(\beta_1,\beta)$, using that $\Sigma_{n+1}\ge\Sigma_n$ is increasing and integrating the previous differential inequality between $\beta'$ and $\beta$ gives
$$T_n (\beta) - T_n(\beta') \geq (\beta - \beta')\,  \frac{\log \Sigma_{n} (\beta')-\log M}{\log n}.$$
Hence, the fact that $T_n(\beta)$ converges to $f(\beta)$ as $n$ tends to infinity implies 
\begin{equation*} f (\beta) -f(\beta') 
~\geq~  (\beta - \beta')  \, \Big[\limsup_{n\rightarrow \infty} \frac{\log \Sigma_n (\beta')}{\log n}\Big]~\geq~ \beta- \beta'.
\end{equation*}
Letting $\beta'$ tend to $\beta_1$ from above, we obtain
$
f(\beta)\ge \beta-\beta_1.
$
\end{proof}

We now present an application of Theorem~\ref{thm:OSSS} to monotonic measures on $\{0,1\}^E$, where $E$ is the edge set of a finite graph $G=(V,E)$. Let $\Lambda_n(x)$ denote the box of size $n$ around $x\in V$ and write $\Lambda_n=\Lambda_n(0)$. We see elements of $\{0,1\}^E$ as percolation configurations and use the corresponding notation.
\begin{lemma}\label{cor:OSSS}
Consider a finite graph $G=(V,E)$ containing $0$. For any monotonic measure $\mu$ on $\{0,1\}^E$ and any $n\ge1$, one has
$$\sum_{xy\in E} {\rm Cov}_\mu(\mathbbm 1_{0\leftrightarrow\partial\Lambda_n},\omega_e)\ge \frac{n}{\displaystyle 4\max_{x\in \Lambda_n}\sum_{k=0}^{n-1}\mu[x\leftrightarrow \partial\Lambda_k(x)]}\cdot \mu[0\leftrightarrow \partial\Lambda_n]\big(1-\mu[0\leftrightarrow \partial\Lambda_n]\big).$$
  \end{lemma}

The proof is based on  Theorem~\ref{thm:OSSS} applied to a well chosen decision tree determining $\mathbbm 1_{0\leftrightarrow\partial\Lambda_n}$. One may simply choose the trivial algorithm checking every edge of the box $\Lambda_n$. Unfortunately, the revealment of this decision tree being 1 for every edge, the OSSS inequality will not bring us more information that the Poincar\'e inequality \eqref{eq:poincare}. A slightly better algorithm would be provided by the decision tree discovering the connected component of the origin ``from inside''. Edges far from the origin would then be revealed by the algorithm if (and only if) one of their endpoints is connected to the origin. This provides a good bound for the revealment of edges far from the origin, but edges close to the origin are still revealed with large probability. In order to avoid this last fact, we will rather choose a family of decision trees discovering the connected components of $\partial\Lambda_k$ for $1\le k\le n$ and observe that the average of their revealment for a fixed edge will always be small.
\begin{figure}
\begin{center}\includegraphics[width=0.65\textwidth]{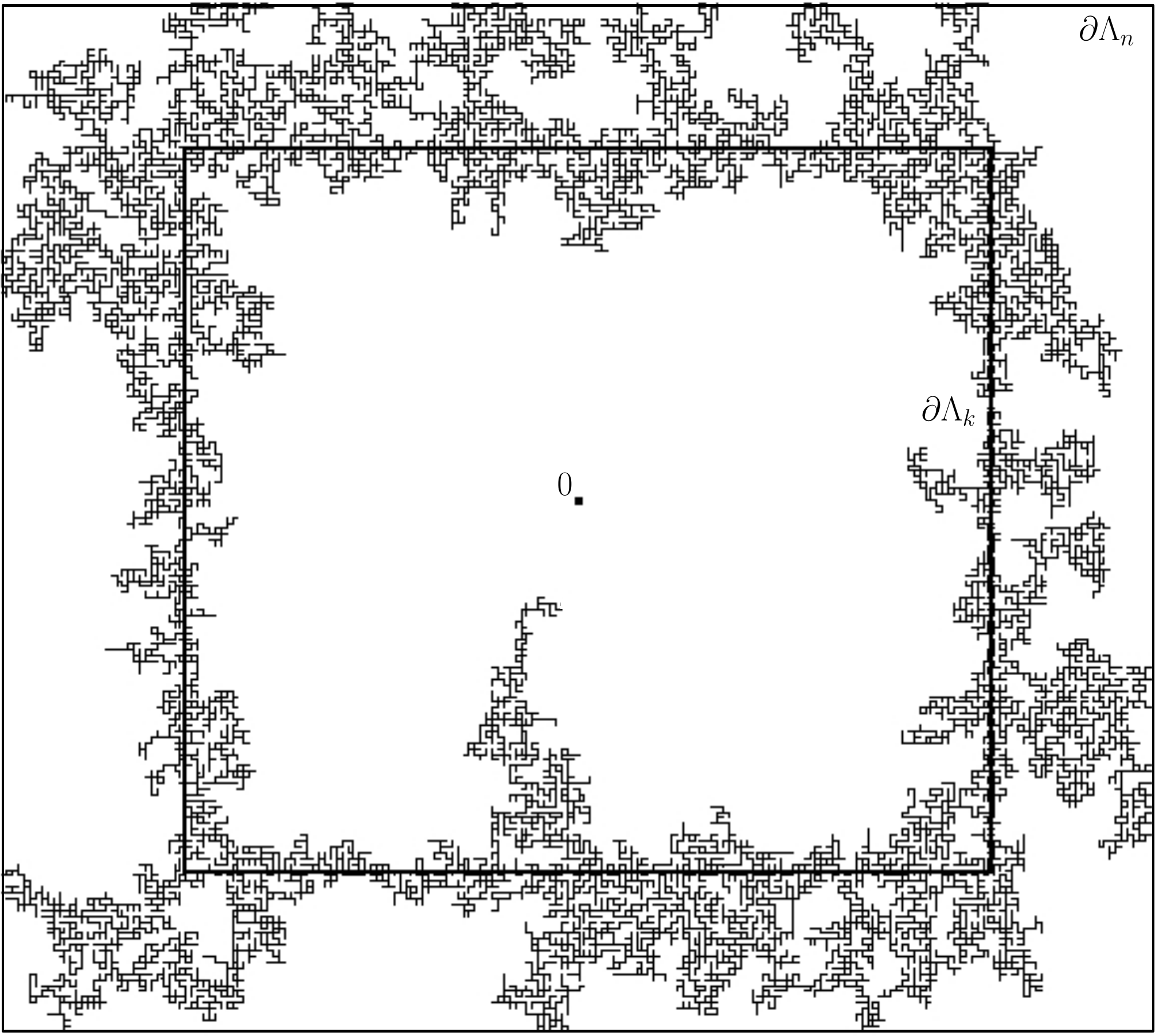}
\caption{\label{fig:algo} A realization of the clusters intersecting $\partial\Lambda_k$. Every edge having one endpoint in this set is potentially revealed by the decision tree before time $\tau$. Furthermore in this specific case, we know that $0$ is not connected to the boundary of $\Lambda_n$. }\end{center}
\end{figure}
\begin{proof} 
  We can assume that $\partial \Lambda_n$ is not empty (otherwise the statement is trivially true). For any $k\in \llbracket 1,n\rrbracket$, we wish to construct a decision tree $T$ determining $\mathbbm 1 _{0\leftrightarrow\partial\Lambda_n}$ such that for each $e=uv$,
\begin{equation} \label{eq:zzz}
\delta_e(T)\le \mu[u\longleftrightarrow \partial\Lambda_k]+\mu[v\longleftrightarrow \partial\Lambda_k].
\end{equation}
Note that this would conclude the proof since we obtain the target inequality by applying Theorem~\ref{thm:OSSS} for each $k$ and then summing on $k$. As a key, we use that for $u\in \Lambda_n$,
\begin{align*}\sum_{k=1}^n\mu[u\longleftrightarrow \partial\Lambda_k]
&~\le~ \sum_{k=1}^n  \mu[u\longleftrightarrow \partial\Lambda_{|k-d(u,0)|}(u)]~\le~ 2\max_{x\in \Lambda_n}\sum_{k=0}^{n-1}\mu[x\leftrightarrow \partial\Lambda_k(x)].
\end{align*}

We describe the decision tree $T$, which corresponds first to an exploration of the connected components in $\Lambda_n$ intersecting $\partial\Lambda_k$ that does not reveal any edge with both endpoints outside these connected components, and then to a simple exploration of the remaining edges.
 
More formally, 
we define $\e$ using two  growing sequences  $\partial\Lambda_k=V_0\subset V_1\subset \cdots\subset V$ and $\emptyset=F_0\subset F_1\subset\cdots\subset F$ (where $F$ is the set of edges between two vertices within distance $n$ of the origin) that should be understood as follows: at step $t$, $V_t$ represents the set of vertices that the decision tree found to be connected to $\partial \Lambda_k$, and $F_t$ is the set of explored edges discovered by the decision tree until time $t$.

Fix an ordering of the edges in $F$. Set $V_0=\partial\Lambda_k$ and $F_0=\emptyset$. Now, assume that $V_t\subset V$ and $F_t\subset F$ have been constructed and distinguish between two cases:\begin{itemize}[noitemsep,nolistsep]
\item If there exists an edge $e=xy\in F\setminus F_t$ with $x\in V_t$ and $y\notin V_t$ (if more than one exists, pick the smallest one for the ordering), then set $\e_{t+1}=e$, $F_{t+1}=F_t\cup\{e\}$ and set $$V_{t+1}:=\begin{cases}V_t\cup\{x\}&\text{ if }\omega_e=1\\ V_t&\text{ otherwise}.\end{cases}$$
\item If $e$ does not exist, set $\e_{t+1}$ to be the smallest $e\in F\setminus F_t$ (for the ordering) and set $V_{t+1}=V_t$ and $F_{t+1}=F_t\cup\{e\}$.\end{itemize}
As long as we are in the first case, we are still discovering the connected components of $\partial\Lambda_k$, while as soon as we are in the second case, we remain in it. The fact that $\tau$ is smaller than  or equal to the last time we are in the first case gives us \eqref{eq:zzz}.
\end{proof}
\begin{remark}
Note that $\tau$ may a priori be strictly smaller than the last time we are in first case (since the decision tree may discover a path of open edges from 0 to $\partial \Lambda_n$ or a family of closed edges disconnecting the origin from $\partial\Lambda_n$ before discovering the whole connected components of $\partial\Lambda_k$).
\end{remark}

We are now in a position to prove Theorem~\ref{thm:RCM}. We will simply combine a derivative formula for random-cluster models with the previous lemma, and then apply Lemma~\ref{lem:technical}.
\begin{proof}[Theorem~\ref{thm:RCM}]
Fix $q\ge1$ and $\beta_0\ge0$. For $n\ge1$ and $\beta\le\beta_0$, define 
$$\mu_n:=\phi_{\Lambda_{2n},\beta,q}^{\rm w}\qquad\qquad\theta_k(\beta):=\mu_k[0\leftrightarrow \partial\Lambda_k]\qquad\qquad S_n:=\sum_{k=0}^{n-1}\theta_k.$$
Now, the comparison between boundary conditions \cite[Lemma 4.14]{Gri06} together with the facts that $\Lambda_{2k}(x)\subset\Lambda_{2n}$ and that $\bbG$ is transitive imply that for $x\in\Lambda_n$,
$$\sum_{k=1}^{n-1}\mu_n[x\leftrightarrow\partial\Lambda_k(x)]\le 2\sum_{k\le n/2}\mu_n[x\leftrightarrow\partial\Lambda_k(x)]\le 2\sum_{k\le n/2}\mu_k[0\leftrightarrow\partial\Lambda_k]\le 2S_n.$$
Since $\mu_n$ is monotonic \cite[Theorem 3.8]{Gri06}, Lemma~\ref{cor:OSSS} (applied to the graph $G=(\Lambda_{2n},E)$ induced by $\Lambda_{2n}$) and the previous bound give
\begin{equation}\label{eq:b}\sum_{e\in E}{\rm Cov}(\mathbbm 1_{0\leftrightarrow \partial\Lambda_n},\omega_e) ~\ge~\frac{n}{8S_n}\cdot \theta_n(1-\theta_n).\end{equation}
Now,
a derivative formula for random-cluster models \cite[Theorem 3.12]{Gri06} implies that
\begin{equation}\label{eq:Russo}\theta_n'(\beta)=\sum_{e\in E}\frac{J_{xy}}{e^{\beta J_{xy}}-1}{\rm Cov}(\mathbbm 1_{0\leftrightarrow \partial\Lambda_n},\omega_e)\ge \min \big\{\frac{J_{xy}}{e^{\beta_0 J_{xy}}-1}\big\} \sum_{e\in E}{\rm Cov}(\mathbbm 1_{0\leftrightarrow \partial\Lambda_n},\omega_e).\end{equation}
Notice that the minimum above is positive (since the coupling constants are finite-range and invariant). Inequalities \eqref{eq:b} and \eqref{eq:Russo} together lead to
 \begin{equation}\label{eq:mlem3}\theta_n'\ge c\,\frac{n}{S_n}\cdot\theta_n,\end{equation}
 where $$c=c(\beta_0):=\frac{1-\theta_1(\beta_0)}{8}\min \Big\{\frac{J_{xy}}{\exp(\beta_0 J_{xy})-1}\Big\}>0$$ (we used that $\theta_n\le\theta_1$ by comparison between boundary conditions and then monotonicity and $\beta\le\beta_0$). 
Measurability implies $\limsup\theta_n=\theta$ while the comparison between boundary conditions gives that $\theta_n\ge \theta$ (for all $n$) so that $\theta_n$ converges to $\theta$.
Lemma~\ref{lem:technical} applied to $f_n=\theta_n/c$ gives the existence of $\beta_1$ such that {\bf P1} and {\bf P2} occur.

Also, for every $n\ge1$,
$$\phi_{\Lambda_{2n},\beta,q}^{\rm w}[0\longleftrightarrow \partial\Lambda_{2n}]\le\theta_n(\beta).$$
Overall, the two previous facts combined with {\bf P1} and {\bf P2} implies the theorem readily (note that when $\beta_c<\infty$, $\beta_1=\beta_c$ as soon as $\beta_0$ is chosen larger than $\beta_c$).
\end{proof}

\section{Proofs of Theorems~\ref{thm:dualRCM} and \ref{thm:dual}}
Without loss of generality, we may assume that $\bbG$ and $\bbG^*$ are embedded in such a way that $\Lambda$ is the set of translations of $\bbZ^2$. We see configurations $\omega$ and $\omega^*$ as subsets of $\bbR^2$ given by the union of the open edges. 
For three sets $A,B,C\subset \bbR^2$,  denote the event that $\omega\cap C$ contains a continuous path from $A$ to $B$ by $A\stackrel{C}\longleftrightarrow B$. 

Let us start by explaining how Theorem~\ref{thm:dualRCM} follows from Theorems~\ref{thm:RCM} and \ref{thm:dual}.

\begin{proof}[Theorem~\ref{thm:dualRCM}]
If $\omega$ has law $\phi_{\bbG,p,q}^{\rm w}$ (we write $p$ in the subscript of the measure instead of $\beta$) and $\omega^*$ is defined by the formula  $\omega^*_{e^*}=1-\omega_e$, then the  duality  \cite[Theorem 6.13]{Gri06} for random-cluster models states that $\omega^*$ has law $\phi_{\bbG^*,p^*,q}^{\rm f}$, where
$$\frac{pp^*}{(1-p)(1-p^*)}= q.$$
 In particular, we need to prove that $p_c(\bbG)^*=p_c(\bbG^*)$. The second item of Theorem~\ref{thm:RCM}, for quasi-transitive graphs, implies that for any $p<p_c(\bbG)$,
$$\sum_{n\ge0}\phi^{\rm w}_{\bbG,p,q}\Big[[n,n+1]\times[0,1]\longleftrightarrow\{0\}\times\bbR\Big]<\infty.$$ 
The Borel-Cantelli lemma implies that there exist only finite circuits of $\omega$ surrounding the origin almost surely. Therefore, by duality, there exists an infinite connected component in $\omega^*$ almost surely, which proves that $p^*\ge p_c(\bbG^*)$. Letting $p$ tend to $p_c(\bbG)$ gives $p_c(\bbG)^*\ge p_c(\bbG^*)$.

On the other hand, ergodic properties of $\phi_{\bbG,p,q}^{\rm w}$ imply that when $p>p_c(\bbG)$, $\omega$ contains a unique infinite connected component almost surely (see \cite{Gri06}). Similarly, if $p^*$ was greater than $p_c(\bbG^*)$, $\omega^*$ would contain a unique infinite connected component almost surely (this uses a known fact \cite{Gri06} that, above the critical point, the random-cluster model with free boundary conditions also contains an infinite connected component almost surely). Therefore, Theorem~\ref{thm:dual} shows that $p>p_c(\bbG)$ implies $p^*\le p_c(\bbG^*)$. Letting $p$ tend to $p_c(\bbG)$  gives $p_c(\bbG)^*\le p_c(\bbG^*)$.
\end{proof}
We now turn to the proof of Theorem~\ref{thm:dual}.
\begin{proof}[Theorem~\ref{thm:dual}]
For $R=[0,n]\times[0,k]$, denote $\mathsf {Top}$, $\mathsf {Left}$, $\mathsf {Bottom}$ and $\mathsf {Right}$ for the top, left, bottom and right sides of the boundary of $R$. Also, define the crossing probabilities
\begin{align*}
v(n,k)&:=\mu[\mathsf {Top}\stackrel{R}\longleftrightarrow\mathsf {Bottom}]\qquad \text{and}\qquad h(n,k):=\mu[\mathsf {Left}\stackrel{R}\longleftrightarrow\mathsf {Right}].
\end{align*}
\begin{lemma}\label{lem:hg}
  Assume that both $\omega$ and $\omega^*$ contain a unique infinite connected components almost surely. Then, as $\min\{n,k\}$ tends to infinity,
\begin{itemize}[noitemsep,nolistsep]
\item $\max\big\{h(n,k),v(n,k+1)\big\}$ tends to 1,
\item $\min\big\{v(n,k),h(n,k)\}$ tends to 0.
\end{itemize}
\end{lemma}
%
Before proving this lemma, let us explain how it implies the theorem. For each $n$, let $k_n$ be the largest integer for which $v(n,k_n)\ge h(n,k_n)$ (note that by definition $v(n,k_n+1)<h(n,k_n+1)$). The uniqueness of the infinite connected component easily implies that $k_n$ tends to infinity as $n$ tends to infinity (for each fixed $k$, the probability that both the infinite connected component and the dual infinite connected component cross $[0,n]\times[0,k]$ from top to bottom tends to 1 as $n$ tends to infinity). 

Now, if both $\omega$ and $\omega^*$ contain infinite connected components almost surely, the first item of the previous lemma implies that $h(n,k_n)$ or $v(n,k_n+1)$ tends to 1. This implies that 
$\min\{v(n,k_n), h(n,k_n)\}$ 
or $\min\{v(n,k_n+1),h(n,k_n+1)\}$ tends to 1, leading to a contradiction with the second item. 

\begin{proof}[Lemma~\ref{lem:hg}]
We prove the first item. The second item is implied by the first one (with the roles of $\omega^*$ and $\omega$ exchanged) since $1-v(n,k)$ and $1-h(n,k)$ are the probabilities that $R$ is respectively crossed horizontally and vertically by a path in $\omega^*$.

Fix $n,k$ and $s$ (that should be thought of as satisfying $1\ll s\ll \min\{n,k\}$). Let $S_y$ be the translate of $S:=[0,s]^2$ by $y\in\bbZ^2$. Define $x=x(R)\in R\cap\bbZ^2$ such that there exists $x'$ and $x''$ neighbors of $x$ in $\bbZ^2$ satisfying
\begin{align}
&\mu[S_x\stackrel{R}\longleftrightarrow \mathsf{Bottom}]\ge \mu[S_x\stackrel{R}\longleftrightarrow \mathsf{Top}],\label{eq:ggg}\\
&\mu[S_x\stackrel{R}\longleftrightarrow \mathsf{Left}]\ge \mu[S_x\stackrel{R}\longleftrightarrow \mathsf{Right}],\label{eq:g11}\\
&\mu[S_{x'}\stackrel{R}\longleftrightarrow \mathsf{Top}]\ge\mu[S_{x'}\stackrel{R}\longleftrightarrow \mathsf{Bottom}],\label{eq:gggg}\\
&\mu[S_{x''}\stackrel{R}\longleftrightarrow \mathsf{Right}]\ge \mu[S_{x''}\stackrel{R}\longleftrightarrow \mathsf{Left}]\label{eq:gg}.
\end{align}
In order to see that this point exists, let $X$ be the set of $x\in \bbZ^2\cap R$ such that \eqref{eq:ggg} holds and denote its boundary in $R$ (i.e.~the set of points in $X$ with one neighbor in $R\setminus X$) by $\partial X$. Let $Y$ and $\partial Y$ be defined similarly with \eqref{eq:g11} instead of \eqref{eq:ggg}. (The sets $X$ and $Y$ are illustrated on Fig.~\ref{fig:xdef}.)  Note that $\partial X\cap\partial Y\ne \emptyset$ since  $\partial X$ contains a path of neighboring vertices crossing $R$ from left to right, and $\partial Y$ a path from top to bottom. By definition, any point in $\partial X\cap \partial Y$ satisfies the property above. 

\paragraph{Claim.} {\em The distance between $x(R)$ and the boundary of $R$ is tending to infinity as $\min\{n,k\}$ tends to infinity. }
\bigbreak
Before proving the claim, let us show how to finish the proof. Let $A_R$ be the event that there is a unique connected component in $\omega\cap R$ going from distance 2 of $S_{x}$ to the boundary of $R$. 
\medbreak
Assume that $\mu[S_x\stackrel{R}\longleftrightarrow \mathsf{Bottom}]\ge \mu[S_x\stackrel{R}\longleftrightarrow \mathsf{Left}].$
 The FKG inequality together with \eqref{eq:ggg} and \eqref{eq:g11} imply that 
\begin{equation}\label{eq:pp}\mu[S_x\stackrel{R}{\not\longleftrightarrow} \mathsf{Bottom}]\le \mu[S\not\longleftrightarrow \infty]^{1/4}.\end{equation} 
Now, set $R'=R+(1,0)$ and $\mathsf{Top}'$ for the top side of $R'$. We find
\begin{align*}
 \mu[S_{x'+(0,1)}\stackrel{R'}\longleftrightarrow \mathsf{Top}']&\stackrel{\phantom{\eqref{eq:ggg}}}=\mu[S_{x'}\stackrel{R}\longleftrightarrow \mathsf{Top}]\\
&\stackrel{\eqref{eq:gggg}}\ge \mu[S_{x'}\stackrel{R}\longleftrightarrow \mathsf{Bottom}]\\
 &\stackrel{\phantom{\eqref{eq:ggg}}}\ge \mu[\{S_{x'}\longleftrightarrow\infty\}\cap\{S_x\stackrel{R}\longleftrightarrow \mathsf{Bottom}\}\cap A_R]\\
 &\stackrel{\eqref{eq:pp}}\ge \mu[S\longleftrightarrow \infty]-\mu[S\not\longleftrightarrow \infty]^{1/4}-\mu[A_R^c].\end{align*}
 We deduce that
 \begin{align*}
 v(n,k+1)&\stackrel{\phantom{\eqref{eq:pp}}}\ge \mu[\{S_x\stackrel{R}\longleftrightarrow \mathsf{Bottom}\}\cap\{S_{x'+(0,1)}\stackrel{R'}\longleftrightarrow \mathsf{Top}'\}\cap A_R]\\
& \stackrel{\eqref{eq:pp}}\ge \mu[S\longleftrightarrow \infty]-2\mu[S\not\longleftrightarrow \infty]^{1/4}-2\mu[A_R^c].
 \end{align*}
 Assume now that $\mu[S_x\stackrel{R}\longleftrightarrow \mathsf{Bottom}]< \mu[S_x\stackrel{R}\longleftrightarrow \mathsf{Left}],$ the same reasoning as above, but using $x''$ instead of $x'$ and \eqref{eq:gg} instead of \eqref{eq:gggg}, leads to the same bound as above for $h(n,k)$.
 The uniqueness of the infinite connected component together with the claim imply that $\mu[A_R]$ tends to 1 as $\min\{n,k\}$ tends to infinity. Letting the size $s$ of $S$ tend to infinity finishes the proof of the first item. To conclude the whole proof, we need to prove the claim.

\begin{figure}[htp]
\centering
\hfill
\begin{minipage}[t]{.44\linewidth}
\centering
  \includegraphics[width=.9\linewidth]{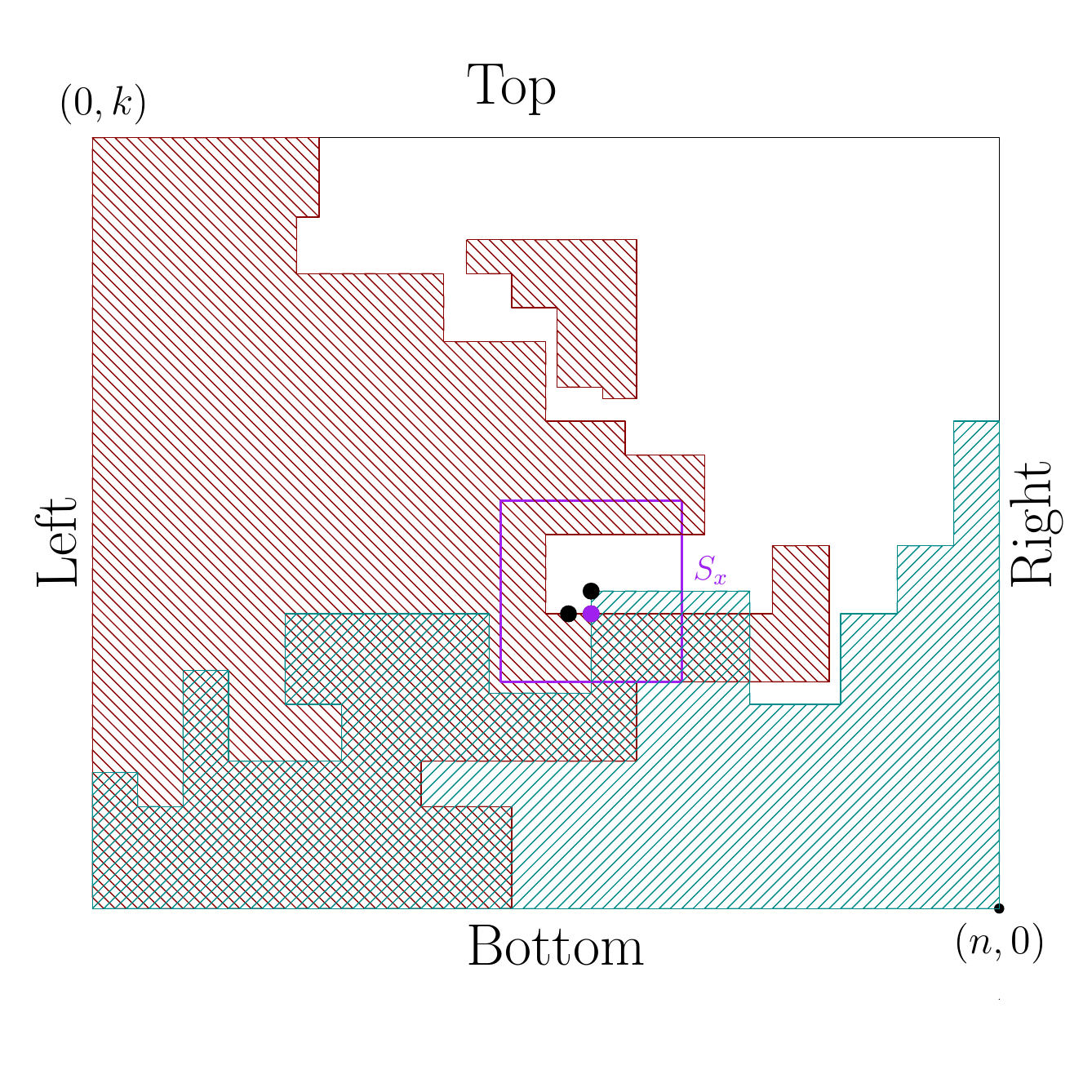}
  \caption{The vertices hatched in green are those  such that $\mu[S_x\stackrel{R}\longleftrightarrow \mathsf{Bottom}]\ge \mu[S_x\stackrel{R}\longleftrightarrow \mathsf{Top}]$, and the vertices hatched in red are those which satisfy  $\mu[S_x\stackrel{R}\longleftrightarrow \mathsf{Left}]\ge \mu[S_x\stackrel{R}\longleftrightarrow \mathsf{Right}]$. The point is selected in the intersection of the boundary of the two regions.}
  \label{fig:xdef}
\end{minipage}
\hfill
\hfill
\begin{minipage}[t]{.44\linewidth}
\centering
  \includegraphics[width=0.9\linewidth]{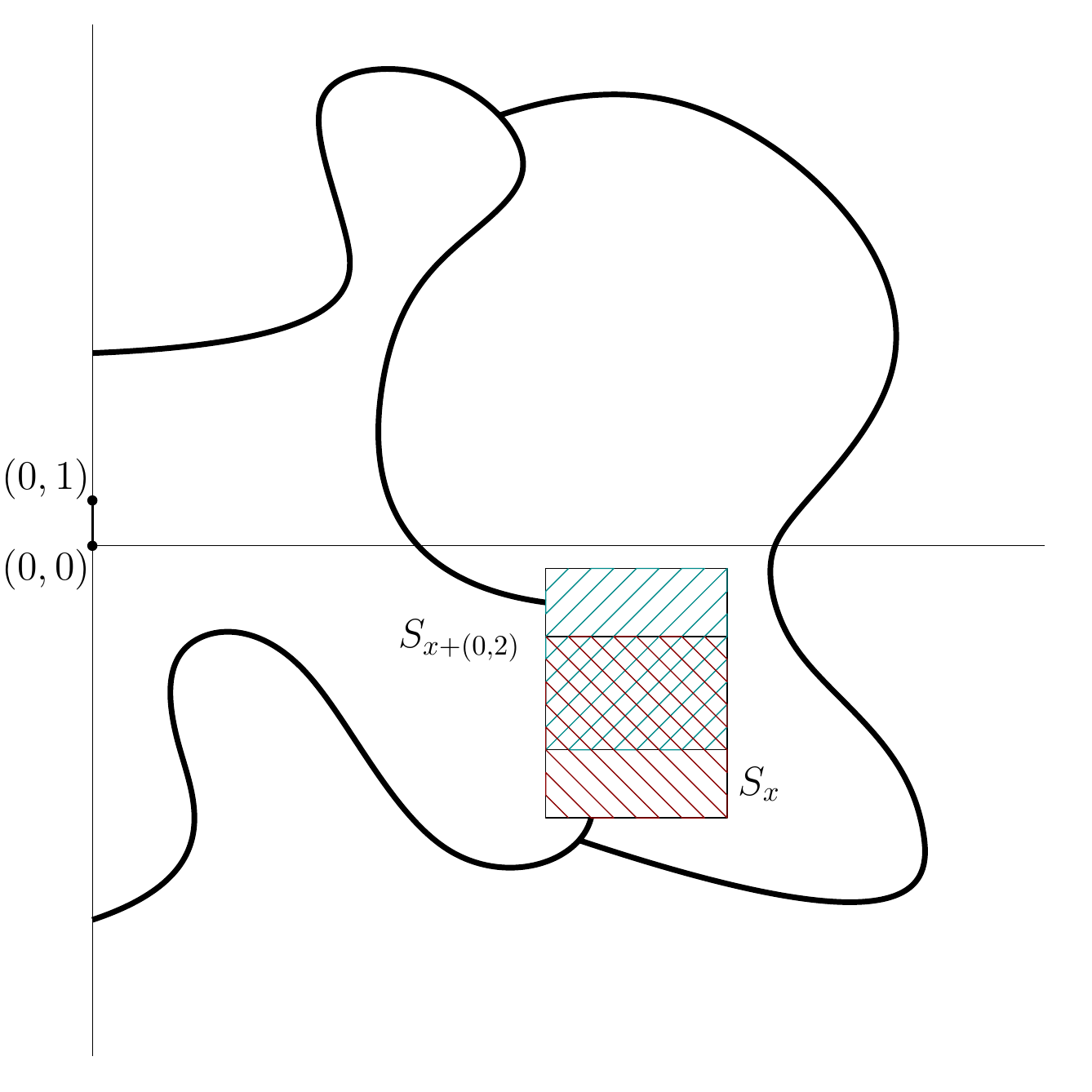}
  \caption{The construction of Claim 1. The black path prevents segment $\{0\}\times[0,1]$ from belonging to an infinite cluster of the dual.}
  \label{fig:half-plane}
\end{minipage}
\hfill
\end{figure}

\paragraph{Proof of the claim}
We prove that the distance to $\mathsf{Left}$ is tending to infinity (the other sides work the same). Note that it is sufficient to prove that $\omega\cap \bbH$, where $\bbH=\bbR_+\times\bbR$, does not contain any infinite connected component almost surely. To avoid introducing new notation, we prove the equivalent statement for $\omega^*\cap\bbH$ instead of $\omega\cap\bbH$, but the proof is the same. Introduce $\ell_+:=\{0\}\times\bbR_+$, $\ell_-:=\{0\}\times\bbR_-$ and $\ell=\ell_-\cup\ell_+$.

For an integer $s, m\ge0$, choose $x=x(m)$ with first coordinate equal to $m$ satisfying
\begin{align*}&\mu[S_x\stackrel{\bbH}\longleftrightarrow \ell_-]\ge\mu[S_x\stackrel{\bbH}\longleftrightarrow \ell_+]\quad\text{ and }\quad\mu[S_{x+(0,1)}\stackrel{\bbH}\longleftrightarrow \ell_-]\le\mu[S_{x+(0,1)}\stackrel{\bbH}\longleftrightarrow \ell_+].
\end{align*}
(This point exists since $\mu[S_x\stackrel{\bbH}\longleftrightarrow \ell_\pm]$ increases to $\mu[S_x\stackrel{\bbH}\longleftrightarrow\ell]$ as the second coordinate of $x$ tends to $\pm\infty$.) The FKG inequality together with these two inequalities implies that 
$$\mu[S_x\stackrel{\bbH}{\longleftrightarrow} \ell_-]\ge1-\sqrt{\mu[S_x{\not\longleftrightarrow}\, \ell]}\quad\text{ and }\quad\mu[S_{x+(0,1)}\stackrel{\bbH}{\longleftrightarrow} \ell_+]\ge1-\sqrt{\mu[S_{x+(0,1)}{\not\longleftrightarrow} \,\ell]}.$$
Let $A_m$ be the event that there is a unique connected component in $\omega\cap\bbH$ going from distance 2 of $S_x$ to $\ell$. Let $B$ be the event that $\omega\cap\bbH$ does not contain an infinite connected component intersecting $\{0\}\times[0,1]$. We find
\begin{align*}
\mu[B]&\stackrel{\phantom{\eqref{eq:pp}}}\ge \mu[\{S_x\stackrel{\bbH}{\longleftrightarrow} \ell_-\}\cap\{S_{x+(0,2)}\stackrel{\bbH}{\longleftrightarrow} \ell_+ + (0,1)\}\cap A_m]\\
&\stackrel{\phantom{\eqref{eq:pp}}}\ge 1-\sqrt{\mu[S_x{\not\longleftrightarrow} \ell]}-\sqrt{\mu[S_{x+(0,1)}{\not\longleftrightarrow} \ell]}-\mu[A_m^c].
\end{align*}
(The construction leading to the bound above is illustrated on Fig.~\ref{fig:half-plane}.)
The uniqueness of the infinite connected component in $\omega^*$ implies that $\mu[A_m]$ tends to 1 as $m$ tends to infinity, and also that for any $y\in \bbH$,
\begin{equation}\label{eq:ppp}\mu[S_y\longleftrightarrow \ell]\ge\mu[\{S_y\longleftrightarrow \infty\}\cap\{S_{-y}\longleftrightarrow\infty\}]\stackrel{\rm FKG}\ge \mu[S\longleftrightarrow\infty]^2.\end{equation}
Letting $m$ tend to infinity and then the size $s$ of $S$ tend to infinity implies that $\mu[B]=1$. This concludes the proof.
\end{proof}  

\newcommand{\etalchar}[1]{$^{#1}$}

\end{document}